\documentclass[a4paper,12pt]{amsart}
\usepackage{amsmath,amsthm,amssymb,hyperref,graphicx,mathrsfs,mathtools}  
\newtheorem{theorem}{Theorem}[section]
\newtheorem{corollary}{Corollary}[section]
\newtheorem{lemma}{Lemma}[section]

\newtheorem{remark}{Remark}[section]

\allowdisplaybreaks
\numberwithin{equation}{section}
\title[Inequalities for the derivative of Polynomials]{Inequalities for the derivative of Polynomials with restricted zeros}
\author{N. A. Rather$^1$}
\author{Ishfaq Dar$^2$}
\author{A. Iqbal$^3$}
\address{$^{1,2,3}$Department of Mathematics, University of Kashmir, Srinagar-190006, India}
\email{drnisar@uok.edu.in, ishfaq619@gmail.com, Itz.a.iqbal@gmail.com}
\begin{document}
\maketitle
\begin{abstract}
In this paper we shall use the boundary Schwarz lemma of Osserman to obtain some generalisations and refinements of some well known results concerning the maximum modulus of the polynomials with restricted zeros due to Tur\'{a}n, Dubinin and others.
\end{abstract}
\section{\textbf{Introduction}}
Let $\mathcal{P}_n$ denote the class of all algebraic polynomials of the form\\
 $P(z)= \sum\limits_{j=o}^{n}a_jz^j$ of degree $n \geq 1.$ 
It was shown by P. Tur\'{a}n \cite{TR} that if $P\in \mathcal P_{n}$ has all its zeros in $\left|z\right|\leq 1$, then 
\begin{align}\label{a1}
\max_{|z|=1} |P^{\prime}(z)|\geq \frac{n}{2}\max_{|z|=1}|P(z)|. 
\end{align}
\noindent Equality in \eqref{a1} holds for $P(z)=az^{n}+b$, $|a|=|b|=1$.\\
\indent As an extension of \eqref{a1}, Govil \cite{GV} proved that if $P\in\mathcal{P}_n$ and $P(z)$ has all its zeros in $|z|\leq k, k \geq 1$, then
\begin{align}\label{a2}
\max_{|z|=1} |P^{\prime}(z)|\geq \frac{n}{1+k^n}\max_{|z|=1}|P(z)|.
\end{align}
The result is sharp as shown by the polynomial $P(z)=z^{n}+k^{n}$.\\
\indent By involving the minimum modulus of $P(z)$ on $|z|=1$, Aziz and Dawood \cite{AD}, proved under the hypothesis of inequality \eqref{a1} that
 \begin{align}\label{a3}
\max_{|z|=1} |P^{\prime}(z)|\geq \frac{n}{2}\left\{\max_{|z|=1}|P(z)|+\min_{|z|=1}|P(z)|\right\}. 
\end{align}
\noindent Equality in \eqref{a3} holds for $P(z)=az^{n}+b$, $|a|=|b|=1$.\\
\indent In literature, there exist several generalizations and extensions of \eqref{a1}, \eqref{a2} and \eqref{a3} (see \cite{AA}-\cite{ARR},\cite {MMR}, \cite{QRS}, \cite{RS}). Dubinin \cite{DB} obtain a refinement of \eqref{a1} by involving some of the coefficients of polynomial $P\in\mathcal{P}_n$ in the bound of inequality \eqref{a1}. More precisely, proved that if all the zeros of the polynomial $P\in\mathcal P_{n}$ lie in $|z|\leq 1$, then
\begin{align}\label{a4}
\max_{|z|=1} |P^{\prime}(z)|\geq \frac{1}{2}\bigg(n+\frac{|a_n|-|a_0|}{|a_n|+|a_0|}\bigg)\max_{|z|=1}|P(z)|.
\end{align}
\section{Main results}
In this paper, we are interested in estimating the lower bound for the maximum modulus of $P^{\prime}(z)$ on $|z|=1$ for $P\in\mathcal P_{n}$  not vanshing in the region $|z|>k$ where $k\geq 1$ and establish some refinements and generalizations of the inequalities \eqref{a1}, \eqref{a2}, \eqref{a3} and \eqref{a4}. We begin by proving:
\begin{theorem}\label{th1}
If all the zeros of polynomial $P\in\mathcal{P}_n$ of degree $n\geq 2$ lie in $|z|\leq k, k\geq 1,$ then
\begin{align}\label{e1}
\begin{split}
\max_{|z|=1}|P^{\prime}(z)|\geq {}& \frac{1}{1+k^{n}}\bigg(n+\frac{k^{n}|a_n|-|a_0|}{k^{n}|a_n|+|a_0|}\bigg) \max_{|z|=1} |P(z)|\\+{}&
\frac{|a_{n-1}|}{k(1+k^{n})}\bigg(n+\frac{k^{n}|a_n|-|a_0|}{k^{n}|a_n|+|a_0|}\bigg)\phi(k)+ |a_{1}|\psi(k),
\end{split}
\end{align}
where $\phi(k)= \left(\frac{k^{n}-1}{n}-\frac{k^{n-2}-1}{n-2}\right)$ or $\frac{(k-1)^{2}}{2}$ and  $\psi(k) =(1-1/k^2) ~\textnormal{or} ~ (1-1/k)$ according as $n>2$ or $n=2.$
\end{theorem}
\noindent The result is best possible and equality in \eqref{e1} holds for $P(z)= z^{n}+k^{n}$.
\begin{remark}\label{rm1} 
\textnormal{Since all the zeros of $P(z)$ lie in $|z| \leq k$ where $k\geq 1$, therefore,
 $|a_0| \leq k^n|a_n|.$ In view of this, inequality \ref{e1} constitutes a refinement of inequality \eqref{a2}. Further, inequality \eqref{e1} reduces to inequality \eqref{a4} for $k =1$.}
\end{remark}
\begin{theorem}\label{th2}
If all the zeros of polynomial $P\in\mathcal{P}_n$ of degree $n\geq 2$ lie in $|z|\leq k$ where $k\geq 1$ and $m=\min_{|z|=k}|P(z)|,$ then for $0 \leq l < 1$
\begin{align}\label{e2}
\begin{split}
\max_{|z|=1} |P^{\prime}(z)|{}&\geq \frac{n}{1+k^n} \big( \max_{|z|=1}|P(z)| + lm \big)+\psi(k)|a_1|\\
&+\frac{1}{k^n(1+k^n)} \bigg\{\bigg(\frac{k^n|a_n|-lm-|a_0|}{k^n|a_n|-lm+|a_0|}\bigg)\big(k^n\max_{|z|=1}|P(z)|-lm\big)\\
&+ k^{n-1}|a_{n-1}|\phi(k) \bigg(n+\frac{k^n|a_n|-lm-|a_0|}{k^n|a_n|-lm+|a_0|}\bigg) \bigg\},\\
\end{split}
\end{align}
\noindent where $\phi(k)$ and $\psi(k)$ are same as defined in Theorem \ref{th1}.
\end{theorem}
\noindent The result is sharp and equality in \eqref{e2} holds for $P(z)= z^{n}+k^{n}$.
\begin{remark}\label{rm2}
\textnormal{ As before, it can be easily seen that Theorem \ref{th2} is a refinement of Theorem \ref{th1}. Moreover, for $k=1$, we get the following result which includes a refinement of inequality \eqref{a4} as a special case.}
\end{remark}
\begin{corollary}\label{c1}
If all the zeros of $P\in\mathcal{P}_n$ of degree $n\geq 2,$ lie in $|z|\leq 1$ and $ m_1=\min_{|z|=1}|P(z)|,$ then for $~0 \leq l < 1$
\begin{align}\label{e5}
\begin{split}
\max_{|z|=1}|P^{\prime}(z)| \geq{}& \frac{n}{2}\left\{\max_{|z|=1}|P(z)|+lm_1 \right\}\\
&+\frac{1}{2}\bigg(\frac{|a_n|-lm_1-|a_0|}{|a_n|-lm_1+|a_0|}\bigg)\big(\max_{|z|=1}|P(z)|-lm_1\big),
\end{split}
\end{align}
\end{corollary}
\noindent The result is sharp and equality holds for $P(z)=(z^{n}+1)$.
\section{\textbf{Lemmas}}
\indent For the proof of these theorems, we need the following lemmas. The first Lemma is due to P. Erd\"{o}s and P. D. Lax \cite{LAX}
\begin{lemma}\label{lmel}
If $P\in \mathcal P_{n}$  does not vanish in $|z|<1$, then
\begin{align}\label{a0}
\max_{|z|=1} |P^{\prime}(z)|\leq \frac{n}{2}\max_{|z|=1}|P(z)|. 
\end{align}
\end{lemma}
Next Lemma is a special case of a result due to Aziz and Rather\cite{ANR,AR}. 
\begin{lemma}\label{lm3}
If $P\in\mathcal{P}_n$ and $P(z)$ has its all zeros in $|z|\leq 1$ and $Q(z)=z^n \overline{P\big({1}/{\overline{z}}\big)},$ then for $|z|=1$,
\begin{align*}
|Q^{\prime}(z)|\leq |P^{\prime}(z)|.
\end{align*}
\end{lemma}
\noindent The following result is due to Frappier, Rahman and Ruscheweyh \cite{FRR}. 
\begin{lemma}\label{lm0}
If $P\in\mathcal{P}_n$ is a polynomial of degree $n\geq 1$, then for $R\geq 1$,
\begin{align}\label{eq3}
\max_{|z|=R}\left|P(z)\right|\leq R^{n}\max_{|z|=1}\left|P(z)\right|-(R^{n}-R^{n-2})|P(0)| \qquad \textnormal{if} \quad n > 1
\end{align}
\textnormal{and} 
\begin{align}\label{eq4}
\max_{|z|=R}\left|P(z)\right|\leq R\max_{|z|=1}\left|P(z)\right|-(R-1)|P(0)| \qquad \textnormal{if} \quad n = 1.
\end{align}
\end{lemma}
\noindent From above lemma, we deduce:
\begin{lemma}\label{lm00}
If $P\in\mathcal{P}_n=a_{n}\prod_{j=1}^{n}(z-z_j)$ is a polynomial of degree $n \geq 2$ having no zeros in $|z|<1$, then for every $\alpha \in \mathbb{C}$ with $|\alpha|\leq 1$ and $R\geq 1$,
\begin{align}\label{eqf}
\begin{split}
\max_{|z|=R}\left|P(z)\right|\leq {}& \frac{R^{n}+1}{2}\max_{|z|=1}\left|P(z)\right|
-|\alpha|\frac{R^{n}-1}{2}\min_{|z|=1}\left|P(z)\right|\\-{}&\left(\frac{R^{n}-1}{n}-\frac{R^{n-2}-1}{n-2}\right)|P^{\prime}(0)| \qquad \textnormal{if} \quad n> 2 
\end{split}
\end{align} 
\textnormal{and} 
\begin{align}\label{eqg}
\begin{split}
\max_{|z|=R}\left|P(z)\right|\leq {}& \frac{R^{2}+1}{2}\max_{|z|=1}\left|P(z)\right|\\
-{}&|\alpha|\frac{R^{2}-1}{2}\min_{|z|=1}\left|P(z)\right|-\frac{(R-1)^{2}}{2}|P^{\prime}(0)| \qquad \textnormal{if} \quad n = 2.
\end{split}
\end{align}
\end{lemma}
\begin{proof}[\bf{Proof of Lemma} \ref{lm00}]  By hypothesis all the zeros of $P(z)$ lie in $|z| \geq 1$.
Let $m=\min_{|z|=1}|P(z)|$, then $m\leq |P(z)|$ for $|z|=1$. Applying Rouche's theorem, it follows that the polynomial $G(z)=P(z)+\alpha m z^{n}$ has all its zeros in $|z|\geq 1$ for every $\alpha$ with $|\alpha|<1$ (this is trivially true for $m=0.$) 
Now for each $\theta$,~~$0\leq\theta < 2\pi$, we have
\begin{align}
G(Re^{i\theta})-G(e^{i\theta})= \int_{1}^{R} e^{i\theta}G^{\prime}(te^{i\theta})dt.
\end{align}
This gives with the help of \eqref{eq3} of Lemma \ref{lm0} and Lemma \ref{lmel} for $n>2$,
\begin{align*}
\begin{split}
\left|G(Re^{i\theta})-G(e^{i\theta})\right| &{} \leq \int_{1}^{R}|G^{\prime}(te^{i\theta})|dt\\
&{}\leq\frac{n}{2}\left(\int_{1}^{R}t^{n-1}dt\right)\max_{|z|=1}|G(z)|-\int_{1}^{R}\left(t^{n-1}-t^{n-3}\right)dt|G^{\prime}(0)|\\
&{}= \frac{R^{n}-1}{2}\max_{|z|=1}|G(z)|-\left(\frac{R^{n}-1}{n}-\frac{R^{n-2}-1}{n-2}\right)|P^{\prime}(0)|,
\end{split}
\end{align*}
so that for $n>2$ and $0\leq\theta < 2\pi$, we have
\begin{align*}
\begin{split}
\left|G(Re^{i\theta})\right|\leq &{} \left|G(Re^{i\theta})-G(e^{i\theta})\right|+\left|G(e^{i\theta})\right|\\
&{}= \frac{R^{n}+1}{2}\max_{|z|=1}|G(z)|-\left(\frac{R^{n}-1}{n}-\frac{R^{n-2}-1}{n-2}\right)|P^{\prime}(0)|.
\end{split}
\end{align*}
Replacing $G(z)$ by $P(z)+\alpha m z^{n}$, we get for $|z|=1$,
\begin{align}\label{ex}
\begin{split}
|P(Rz)+\alpha m R^{n}z^{n}|\leq {}&\frac{R^{n}+1}{2}\max_{|z|=1}|P(z)+\alpha m z^{n}|
\\{}&-\left(\frac{R^{n}-1}{n}-\frac{R^{n-2}-1}{n-2}\right)|P^{\prime}(0)|.
\end{split}
\end{align}
Choosing argument of $\alpha$ in the left hand side of \eqref{ex} suitably, we obtain for $n>2$ and $|z|=1$,
\begin{align*}
\begin{split}
|P(Rz)|+|\alpha| m R^{n}\leq {}&\frac{R^{n}+1}{2}\left\{\max_{|z|=1}|P(z)|+|\alpha|m\right\}\\{}&
-\left(\frac{R^{n}-1}{n}-\frac{R^{n-2}-1}{n-2}\right)|P^{\prime}(0)|,
\end{split}
\end{align*}
equivalently for $n>2$, $|\alpha|<1$ and $|z|=1$, we have
\begin{align*}
\begin{split}
|P(Rz)|\leq {}&\frac{R^{n}+1}{2}\max_{|z|=1}|P(z)|-|\alpha|\frac{R^{n}-1}{2}\min_{|z|=1}|P(z)|\\{}&
-\left(\frac{R^{n}-1}{n}-\frac{R^{n-2}-1}{n-2}\right)|P^{\prime}(0)|,
\end{split}
\end{align*}
which proves inequality \eqref{eqf} for $n>2$ and $|\alpha|<1$. Similarly we can prove inequality \eqref{eqg} for $n=2$ by using \eqref{eq4} of Lemma \ref{lm0} instead of \eqref{eq3}. For $|\alpha|=1$, the result follows by continuity. This completes the proof of Lemma \ref{lm00}.
\end{proof}
Finally we also need the Lemma due to Osserman \cite{RO}, known as boundary Schwarz lemma.
\begin{lemma}\label{lm5}
If\\
(a) \qquad $f(z)$ is analytic for $|z|<1$,\\
(b) \qquad $|f(z)|<1$ for $|z|<1$,\\
(c) \qquad $f(0)=0$,\\
(d) \qquad for some $b$ with $|b|=1, f(z)$ extends continuously to $b,$\\
\indent \qquad \quad $|f(b)|=1$ and $f^{\prime}(b)$ exists.\\
Then 
\begin{align}\label{le2}
|f^{\prime}(b)|\geq \frac{2}{1+|f^{\prime}(0)|}.
\end{align}
\end{lemma}
\section{\textbf{Proof of the Theorems}}
\begin{proof}[\textbf{Proof of Theorem \ref{th1}}] Let $g(z)=P(kz).$ Since all the zeros of $P(z)=a_{n}\prod_{j=1}^{n}(z-z_{j})$ lie in $|z|\leq k$ where $k\geq 1$, $g(z)$ has all its zeros in $|z|\leq 1$ and hence all the zeros of the conjugate polynomial $g^*(z)= z^n\overline{g(1/\bar{z})}$ lie in $|z|\geq 1.$\\
Therefore, the function 
\begin{align}\label{e}
F(z)=\frac{g(z)}{z^{n-1}\overline{g(1/\overline{z})}}=z\frac{a_n}{\bar{a_{n}}}\prod_{j=1}^{n}\left(\frac{kz-z_{j}}{k-z\bar{z_{j}}}\right) 
\end{align}
is analytic in $|z|<1$ with $F(0)=0$ and $|F(z)|=1$ for $|z|=1.$
Further for $|z|=1$, this gives
\begin{align*}
\frac{zF^{\prime}(z)}{F(z)}=1-n+\frac{zg^{\prime}(z)}{g(z)}+\overline{\bigg(\frac{zg^{\prime}(z)}{g(z)}\bigg)}
\end{align*}
so that
\begin{align}\label{thp1}
Re \bigg(\frac{zF^{\prime}(z)}{F(z)}\bigg) =1-n+2 Re\bigg(\frac{zg^{\prime}(z)}{g(z)}\bigg).
\end{align}
Also, we have from \eqref{e}
\begin{align*}
\frac{zF^{\prime}(z)}{F(z)}=1+\sum_{j=1}^{n}\left(\frac{k^{2}-|z_{j}|^2}{|kz-z_{j}|^{2}}\right)>0~~~\textnormal{for} ~~~ |z|=1,
\end{align*}
as such,
\begin{align*}
\frac{zF^{\prime}(z)}{F(z)}=\left|\frac{zF^{\prime}(z)}{F(z)}\right|=|F^{\prime}(z)|\qquad \textnormal{for} \quad |z|=1.
\end{align*}
Using this fact in \eqref{thp1}, we get
for points $z$ on $|z|=1$ with $g(z) \neq 0,$
\begin{align}\label{thp2}
1-n+2Re\bigg(\frac{zg^{\prime}(z)}{g(z)}\bigg)= |F^{\prime}(z)|.
\end{align}
Applying lemma \ref{lm5} to $F(z)$, we obtain for all points $z$ on $|z|=1$ with $g(z) \neq 0,$
\begin{align*}
1-n+2 Re\bigg(\frac{zg^{\prime}(z)}{g(z)}\bigg) \geq \frac{2}{1+|F^{\prime}(0)|},
\end{align*}
that is, for $|z|=1$ with $g(z) \neq 0,$
\begin{align*}
Re\bigg(\frac{zg^{\prime}(z)}{g(z)}\bigg) \geq  \frac{1}{2}\bigg(n+\frac{k^n|a_n|-|a_0|}{k^n|a_n|+|a_0|}\bigg).
\end{align*}
This implies
\begin{align*}
\bigg|\frac{zg^{\prime}(z)}{g(z)}\bigg| \geq \frac{1}{2}\bigg(n+\frac{k^n|a_n|-|a_0|}{k^n|a_n|+|a_0|}\bigg) \qquad \textnormal{for ~  $|z|=1, ~ g(z) \neq 0,$}
\end{align*}
and hence,
\begin{align}\label{ne}
|g^{\prime}(z)| \geq \frac{1}{2}\bigg(n+\frac{k^n|a_n|-|a_0|}{k^n|a_n|+|a_0|}\bigg)|g(z)| \qquad \textnormal{for ~ $|z|=1.$}
\end{align}
Replacing $g(z)$ by $P(kz),$ we get for $|z|=1,$
\begin{align*}
k|P^{\prime}(kz)| \geq \frac{1}{2}\bigg(n+\frac{k^n|a_n|-|a_0|}{k^n|a_n|+|a_0|}\bigg)|P(kz)|,
\end{align*}
or equivalently,
\begin{align}\label{pth4}
2k\max_{|z|=k} |P^{\prime}(z)| \geq \bigg(n+\frac{k^n|a_n|-|a_0|}{k^n|a_n|+|a_0|}\bigg) \max_{|z|=k} |P(z)|.
\end{align}
Since $P^{\prime}(z)$ is a polynomial of degree $n-1$, by \eqref{eq3} of Lemma \ref{lm0} with $R=k\geq 1$, we have
\begin{align*}
k^{n-1}\max_{|z|=1}|P^{\prime}(z)|-(k^{n-1}-k^{n-3})|a_1|\geq \max_{|z|=k}|P^{\prime}(z)|, \qquad \textnormal{if} \quad n >2.
\end{align*} 
Combining this inequality with \eqref{pth4}, we get for $n>2,$
\begin{align}\label{pth5}
\begin{split}
2k^{n}\max_{|z|=1}|P^{\prime}(z)|-2(k^{n}-k^{n-2})|a_{1}|\geq\bigg(n+\frac{k^n|a_n|-|a_0|}{k^n|a_n|+|a_0|}\bigg) \max_{|z|=k} |P(z)|.
\end{split}
\end{align} 
Since all the zeros of polynomial $g^{*}(z)= z^{n}\overline{g(1/\bar{z})}=z^{n}\overline{P(k/\bar{z})}$ lie in $|z|\geq 1,$
applying \eqref{eqf} of Lemma \ref{lm00} with $R=k\geq 1$ and $\alpha=0$ to the polynomial $g^{*}(z)$,we get
\begin{align*}
\begin{split}
\max_{|z|=k}|g^{*}(z)|\leq {}&\frac{k^{n}+1}{2}\max_{|z|=1}|g^{*}(z)|\\{}&-\frac{|a_{n-1}|}{k}\left(\frac{k^{n}-1}{n}-\frac{k^{n-2}-1}{n-2}\right)~~\qquad \textnormal{if}\quad n >2.
\end{split}
\end{align*}
That is,
\begin{align*}
\begin{split}
k^{n}\max_{|z|=1}|P(z)|\leq {}&\frac{k^{n}+1}{2}\max_{|z|=k}|P(z)|\\{}&-|a_{n-1}|k^{n-1}\left(\frac{k^{n}-1}{n}-\frac{k^{n-2}-1}{n-2}\right)~~\qquad \textnormal{if} \quad n >2,
\end{split}
\end{align*}
or equivalently, we have for $n>2$,
\begin{align*}
\max_{|z|=k}|P(z)| \geq \frac{2k^{n}}{k^{n}+1}\max_{|z|=1}|P(z)|+\frac {2k^{n-1}|a_{n-1}|}{k^{n}+1}\left(\frac{k^{n}-1}{n}-\frac{k^{n-2}-1}{n-2}\right).
\end{align*} 
Using above inequality in \eqref{pth5}, we get for $n>2$, 
\begin{align*}
\begin{split}
2k^{n}\max_{|z|=1}|P^{\prime}(z)|-{}&2(k^{n}-k^{n-2})|a_{1}|\geq \frac{2k^{n}}{1+k^{n}}\bigg(n+\frac{k^n|a_n|-|a_0|}{k^n|a_n|+|a_0|}\bigg) \max_{|z|=1} |P(z)|\\
+{}&\frac{2k^{n-1}|a_{n-1}|}{1+k^{n}}\bigg(n+\frac{k^n|a_n|-|a_0|}{k^n|a_n|+|a_0|}\bigg)\left(\frac{k^{n}-1}{n}-\frac{k^{n-2}-1}{n-2}\right),
\end{split}
\end{align*}
consequently, 
\begin{align*}
\begin{split}
\max_{|z|=1}|P^{\prime}(z)|\geq {}& \frac{1}{1+k^{n}}\bigg(n+\frac{k^{n}|a_n|-|a_0|}{k^{n}|a_n|+|a_0|}\bigg) \max_{|z|=1} |P(z)|\\+{}&
\frac{|a_{n-1}|}{k(1+k^{n})}\bigg(n+\frac{k^{n}|a_n|-|a_0|}{k^{n}|a_n|+|a_0|}\bigg)\left(\frac{k^{n}-1}{n}-\frac{k^{n-2}-1}{n-2}\right)\\
+{}&(1-1/k^{2})|a_{1}| , \qquad \qquad \qquad \qquad \qquad \textnormal{if} \quad n>2,
\end{split}
\end{align*}
which proves inequality \eqref{e1} for the case $n>1.$ For the case $n=2$, the result follows on similar lines in view of part second of Lemma \ref{lm0} and Lemma \ref{lm00} with $\alpha=0$.
This completes the proof of Theorem \ref{th1}.
\end{proof}
\begin{proof}[\textbf{Proof of Theorem \ref{th2}}]
By hypothesis  $P\in\mathcal{P}_n$ and $P(z)$ has all its zeros in $|z|\leq k, k\geq 1$. If $P(z)$ has a zero on $|z|=k$, then $m =0$ and the result follows by Theorem \ref{th1}. Henceforth, we assume that all the zeros of $P(z)$ lie in $|z|<k,$ so that $m>0$. Hence all the zeros of $h(z)=P(kz)$ lie in disk $|z|<1$ and $ m=\min_{|z|=k}|P(z)|=\min_{|z|=1}|h(z)|.$ Therefore, we have $m \leq |h(z)|$ for $|z|=1.$ This implies for every $\lambda \in \mathbb{C}$ with $|\lambda|<1$ that 
\begin{align*}
m |\lambda z^n| < |h(z)| \qquad \textnormal{for} \quad |z|=1.
\end{align*}
Applying Rouche's theorem, it follows that all the zeros of the polynomial
$H(z)=h(z)+\lambda m z^n$ lie in $|z|<1$ for every $\lambda \in \mathbb{C}$ with  $|\lambda|<1$. Now proceeding similarly as in the proof of Theorem \ref{th1} (with $g(z)$ replacing by $H(z)$), we obtain from \eqref{ne}
\begin{align}\label{tpa1}
|H^{\prime}(z)| \geq \frac{1}{2}\bigg(n+ \frac{|k^na_n+\lambda m|-|a_0|}{|k^na_n+\lambda m|+|a_0|}\bigg) |H(z)| \qquad \textnormal{for} ~ |z|=1.
\end{align}
Using the fact that the function $t(x)= \frac{x-|a|}{x+|a|}$ is non-decreasing function of $x$ and $|k^na_n+\lambda m| \geq k^n|a_n|-|\lambda m|$, we get for every $\lambda \in \mathbb{C}$ with $|\lambda|<1$ and $|z|=1,$
\begin{align}\label{tpa2}
|H^{\prime}(z)| \geq \frac{1}{2}\bigg(n+ \frac{k^n|a_n|-|\lambda m|-|a_0|}{k^n|a_n|-|\lambda m|+|a_0|}\bigg) |H(z)|.
\end{align}
Equivalently for $|z|=1$ and $|\lambda|<1,$
\begin{align}\label{tpa3}
|h^{\prime}(z)+nm\lambda z^{n-1}| \geq \frac{1}{2}\bigg(n+ \frac{k^n|a_n|-|\lambda m|-|a_0|}{k^n|a_n|-|\lambda m|+|a_0|}\bigg) (|h(z)|-m|\lambda|).
\end{align}
Since all the zeros of $H(z)=h(z)+\lambda m z^n$ lie in $|z|<1,$ by Guass Lucas theorem it follows that all the zeros of $H^{\prime}(z)=h^{\prime}(z)+\lambda nm z^{n-1}$ lie in $|z|<1$ for every $\lambda \in \mathbb{C}$ with $|\lambda|<1.$ This implies
\begin{align}\label{gl}
|h^{\prime}(z)| \geq nm |z|^{n}\qquad \textnormal{for} ~  |z| \geq 1.
\end{align}
\noindent Choosing argument of $\lambda$ in the left hand side of \eqref{tpa3} such that
\begin{align*}
|h^{\prime}(z)+nm\lambda z^{n-1}| = |h^{\prime}(z)|-nm|\lambda| \qquad \textnormal{for} ~  |z|=1,
\end{align*}
which is possible by \eqref{gl}, we get
\begin{align*}
|h^{\prime}(z)|-nm|\lambda| \geq \frac{1}{2}\bigg(n+ \frac{k^n|a_n|-|\lambda m|-|a_0|}{k^n|a_n|-|\lambda m|+|a_0|}\bigg) (|h(z)|-m|\lambda|),
\end{align*}
that is,
\begin{align*}
\begin{split}
|h^{\prime}(z)| \geq {}&\frac{1}{2}\bigg(n+ \frac{k^n|a_n|-|\lambda m|-|a_0|}{k^n|a_n|-|\lambda m|+|a_0|}\bigg)|h(z)|\\{}&+\frac{1}{2}\bigg(n- \frac{k^n|a_n|-|\lambda m|-|a_0|}{k^n|a_n|-|\lambda m|+|a_0|}\bigg)|\lambda|m.
\end{split}
\end{align*}
Replacing $h(z)$ by $P(kz)$, we get 
\begin{align}\label{tpa4}
\begin{split}
k\max_{|z|=k}|P^{\prime}(z)| \geq{}& \frac{1}{2}\bigg(n+ \frac{k^n|a_n|-|\lambda m|-|a_0|}{k^n|a_n|-|\lambda m|+|a_0|}\bigg)\max_{|z|=k}|P(z)|\\
&+\frac{1}{2}\bigg(n- \frac{k^n|a_n|-|\lambda m|-|a_0|}{k^n|a_n|-|\lambda m|+|a_0|}\bigg)|\lambda|m.
\end{split}
\end{align}
Again as before, using \eqref{eq3} of Lemma \ref{lm0} and \eqref{eqf} of lemma \ref{lm00}, we obtain for $0 \leq l < 1$ and $n>2$,
\begin{align*}
\begin{split}
k^{n}\max_{|z|=1}|P^{\prime}(z)|- {}& (k^{n}-k^{n-2})|a_1| \geq \frac{1}{2}\bigg(n+ \frac{k^n|a_n|-lm-|a_0|}{k^n|a_n|-lm+|a_0|}\bigg)\\
&\times \bigg\{ \frac{2k^{n}}{1+k^{n}}\max_{|z|=1}|P(z)|+l \bigg(\frac{k^{n}-1}{k^{n}+1} \bigg)\min_{|z|=k}|P(z)|\\
&+\frac{2k^{n-1}|a_{n-1}|}{k^n+1} \bigg(\frac{k^n-1}{n}-\frac{k^{n-2}-1}{n-2}\bigg)\bigg\}\\
&+\frac{1}{2}\bigg(n-\frac{k^n|a_n|-lm-|a_0|}{k^n|a_n|-lm+|a_0|}\bigg)lm,
\end{split}
\end{align*}
which on simplification yields for $0 \leq l < 1$ and $ n>2 $,
\begin{align*}
\begin{split}
\max_{|z|=1} |P^{\prime}(z)|{}&\geq \frac{n}{1+k^n} \left( \max_{|z|=1}|P(z)| + lm \right)+\frac{n|a_{n-1}|}{k(1+k^{n})} \bigg( \frac{k^{n}-1}{n}-\frac{k^{n-2}-1}{n-2}\bigg)\\
&+\bigg(\frac{k^n|a_n|-lm-|a_0|}{k^n|a_n|-lm+|a_0|}\bigg)\bigg\{\frac{1}{k^n(1+k^n)} \left(k^n\max_{|z|=1}|P(z)|-lm \right)\\
&+\frac{|a_{n-1}|}{k(1+k^{n})} \bigg( \frac{k^{n}-1}{n}-\frac{k^{n-2}-1}{n-2}\bigg) \bigg\}+(1-1/k^{2})|a_1|.\\
\end{split}
\end{align*}
The above inequality is equivalent to the inequality \eqref{e2} for $n>2$. For $n=2$, the result follows on the similar lines by using inequality \eqref{eq4} of Lemma \ref{lm0} and inequality \eqref{eqg} of Lemma \ref{lm00} in the inequality \eqref{tpa4}. This proves Theorem \ref{th2}.
\end{proof}
\section{\textbf{Concluding Remark}}
\textnormal If we use Lemma \ref{lm0} and  Lemma \ref{lm00} with $|\alpha|=1$ in the proof of Theorem \ref{th1}, we get the following refinement of inequalities \eqref{a2} and \eqref{e1}.
\begin{theorem}\label{th3}
If $P\in\mathcal{P}_n$ has all its zeros in $|z|\leq k$ where $k\geq 1,$ then
\begin{align}\label{e5}
\begin{split}
\max_{|z|=1} |P^{\prime}(z)|{}&\geq \frac{1}{1+k^n}\bigg(n+\frac{k^n|a_n|-|a_0|}{k^n|a_n|+|a_0|}\bigg)\max_{|z|=1}|P(z)|+|a_1|\psi(k)\\
+{}&\frac{k^{n}-1}{2k^{n}(1+k^n)}\bigg(n+\frac{k^n|a_n|-|a_0|}{k^n|a_n|+|a_0|}\bigg)\min_{|z|=1}|P(z)|
\end{split}
\end{align}
\end{theorem}
where 
\begin{align*}
\psi(k)= \left(1-1/k^{2}\right)~ or ~ \left(1-1/k\right)~ according ~as ~ n > 2 ~or ~ n = 2.
\end{align*}
The result is sharp and equality in \eqref{e5} holds for $P(z)= z^{n}+k^{n}$.
.
\end{document}